\documentclass{article}
\usepackage{amsmath,amsthm,amsfonts, amssymb,amscd}
\usepackage[all]{xy}
\usepackage[utf8]{inputenc}
\usepackage{hyperref}
\hypersetup{colorlinks=true,linkcolor=blue,filecolor=magenta,urlcolor=cyan}
\newtheorem{theorem}{Theorem}[section]

\theoremstyle{definition}
\newtheorem{definition}[theorem]{Definition}

\newtheorem{remark}[theorem]{Remark}

 \numberwithin{equation}{section}
\usepackage{microtype}
\begin{document}
\title{\bf{New Boundes for Sombor Index of Graphs}}

%

\author{ Maryam Mohammadi, Hasan Barzegar \footnote{Corresponding author (barzegar@tafreshu.ac.ir)}\\
{\small\em Department of Mathematics, Tafresh University, Tafresh 39518-79611, Iran.} \\
}

\maketitle

\begin{abstract}
 \noindent
In this paper, we find some bounds for the Sombor index of the graph 
$G$ 
by triangle inequality, arithmetic index, geometric index, forgotten index (F(G)), arithmetic-geometric (AG) index, geometric-arithmetic (GA) index, symmetric division deg index (SDD(G)) and some central and dispersion indices. The bounds could state estimated values and error intervals of the Sombor index to show limits of accuracy. The error intervals are written as inequalities.
 \\[2mm]


 \end{abstract}

\section{Introduction}

 A topological index of a graph is a real number pertaining to the graph, independent of the labeling or graphic displaying. Also, it is invariant under graph isomorphism and states some molecular graph's structural properties. One of these topological indices is the Sombor index introduced by Ivan Gutman for the physicochemical properties of molecules and surveyed for some graphs by Saeid Alikhani and Nima Ghanbary, Kinkar Chandra Das, et all, Nima Ghanbari and Saeid Alikhani,  Igor Milovanovi\'{c} and et all, Maryam Mohammadi and et all  in articles such as  \cite{Kin}, \cite{nima}, \cite{Igor} and \cite{mary}.

  Let 
$G=(V,E)$
(
$|V(G)|=n$
and 
$|E(G)|=m$
) be a graph, then the (reduced and averaged) Somber index is defined as below
\begin{align}\label{eq1}
SO(G)=\sum_{{uv}\in E(G)}\sqrt{d^2(u)+d^2(v)},~~~~~~~~~~~~~~~~~~~~~
\end{align}
\begin{align}\label{eq2}
SO_{red}(G)=\sum_{{uv}\in E(G)}\sqrt{(d(u)-1)^2+(d(v)-1)^2}~~~~
\end{align}
and
\begin{align}\label{eq3}
SO_{ave}(G)=\sum_{{uv}\in E(G)}\sqrt{(d(u)-\frac{2m}{n})^2+(d(v)-\frac{2m}{n})^2}.
\end{align}

where
$d(u)$
is the degree of vertex 
$u$
in 
$G.$
\\
Some other indices such as the arithmetic index and geometric index are studied by J.M Aldaz and Halil \.{I}brahim \c{C}elik in articles like \cite{Aldas}, \cite{m-aldas} and \cite{Ibra}.
GA and AG indices are studied by Shu-Yu Cui et all, Dhritikesh Chakrabarty, Kinkar Ch.Das, Ronald E. Glaser, Burt Rodin, Sa\v{s}a Vujo\v{s}evi\'{c} and et all, in articles such as \cite{cui}, \cite{kin-2}, \cite{cha}, \cite{glas}, \cite{Rodin} and \cite{vuj}
and symmetric division deg index (SDD(G)) in reference \cite{vas} by Alexander Vasilyev. 

Here we will state some bounds for the Sombor index by some mathematics relations and famous indices.
\section{The triangle inequality and Sombor index}\label{sec-2}

\begin{definition}
The ordered pair 
$z=(d(u),d(v))$
as a point in degree-coordinate (or d-coordinate) called the degree-point (or d-point) of the edge 
$uv \in E(G)$
such that 
$d(u)$
denote the degree of the vertex 
$u$
and
$d(v)$
the degree of te vertex
$v$
in the (2-dimensional) coordinate system. The point with coordinates
$(d(v),d(u))$
is the dual-degree-point (or dd-point) of the edge
$uv \in E(G).$
\end{definition}

\begin{definition}
The degree-radius (or d-radius) of the edge 
$uv \in E(G)$
is the distance between
$(d(u),d(v))$
and the origin of the coordinate system denoted by
$|z|=r(d(u), d(v))=\sqrt{d^2(u)+d^2(v)},$
and we have
\begin{align}\label{7}
\sqrt{2}|z|\ge d(u)+d(v).
\end{align}
 The distance between two d-points 
$z_1=(d(u_1),d(v_1))$
and
$z_2=(d(u_2),d(v_2))$
is \\
$|z_1-z_2|=\sqrt{(d(u_1)-d(u_2))^2+(d(v_1)-d(v_2))^2}$
and  triangle inequality gives us a upper bound for the absolute of the summation of d-points
$z_1, z_2,...,z_n$
as
$|z_1+z_2+...+z_n|\le |z_1|+|z_2|+...+|z_n|$
for
$n=2,3,...~.$
\end{definition}

\begin{remark}
In cases of the reduced Sombor index and averaged Sombor index, it is enough to change once
$d(u) \rightarrow d(u)-1$
and once
$d(u) \rightarrow d(u)-\frac{2m}{n},$
then we have
$z'=(d(u)-1,d(v)-1)$ 
and
$z''=(d(u)-\frac{2m}{n},d(v)-\frac{2m}{n})$ 
such that 
$|z'|=r(d(u)-1, d(v)-1)=\sqrt{(d(u)-1)^2+(d(v)-1)^2}$
and
$|z''|=r(d(u)-\frac{2m}{n}, d(v)-\frac{2m}{n})=\sqrt{(d(u)-\frac{2m}{n})^2+(d(v)-\frac{2m}{n})^2}.$
Then similarly use the following definitions and theorems . 
\end{remark}

\begin{theorem}\label{theo1}
Let 
$G$
be a connected graph with 
$n$
vertices and 
$m$
edges such that d-point 
$z=(d(u),d(v))$ 
is pertained to the edge 
$uv \in E(G),$
then 
\begin{align}
|z_1+z_2+...+z_m|\le SO(G) \le \sum_{i=1}^m\sqrt{2(|z_i|^2-d(u_i)d(v_i))}.
\end{align} 
\end{theorem}
\begin{proof}
 Base on the triangle inequality and its generalization
$$SO(G)=\sum _{uv \in E(G)}\sqrt{d^2(u)+d^2(v)}=|z_1|+|z_2|+...+|z_m| \ge |z_1+z_2+...+z_m|.$$
For the right inequality, using (\ref{7}), 
$(d(u)+d(v))^2\le (\sqrt{2}|z|)^2$ 
and hence
$d^2(u)+d^2(v)\le 2|z|^2-2d(u)d(v),$
 which implies 
$SO(G)\le \sum_{i=1}^m \sqrt{2(|z_i|^2-d(u_i)d(v_i))}.$
\end{proof}  

\section{The forgotten index and Sombor index}
\begin{definition}
Consider the arithmetic mean for the nonnegative real numbers as
$ \mu=R_a=\frac{x_1+x_2+...+x_m}{m}=\frac{1}{m}\sum_{i=1}^mx_i,$
for the geometric mean as 
$R_g=\sqrt[m]{x_1.x_2...x_m}=\prod_{i=1}^mx_i^\frac{1}{m}$
and for the harmonic mean as
$R_h=\frac{m}{\sum_{i=1}^m{\frac{1}{x_i}}}.$
\end{definition}
\begin{remark}
Note that {\it Inequality of arithmatic and geometric mean} for the nonnegative real numbers
$x_1,~x_2,~...,~x_m$ 
is as follow,
\begin{align}\label{eq7}
R_g=\prod_{i=1}^mx_i^\frac{1}{m}\le \frac{1}{m}\sum_{i=1}^mx_i=R_a.
\end{align}
Equality holds whenever 
$x_1=x_2=...=x_m.$
\\
Also with refering to the refrence \cite{m-aldas} we could consider that for 
$i=1,~2,~...,~m,$
if
$\alpha=(\alpha_1,\alpha_2,...,\alpha_m),$
$\alpha _i >0$
and
$\sum_{i=1}^m\alpha_i=1,$
then the general arithmethic-geometric inequality is
$\prod_{i=1}^mx_i^{\alpha_i}\le \sum_{i=1}^m\alpha^ix^i.$
Also if apply the variable change 
$x_i=y_i^s,$
then the general arithmethic-geometric inequality for 
$s>0$
is as follows:
\begin{align}\label{eqq7}
\prod_{i=1}^my_i^{\alpha_i} \le(\sum_{i=1}^m\alpha_i y_i^s)^\frac{1}{s},
\end{align}
and for 
$0<s<1,$
\begin{align}\label{eq8}
(\sum_{i=1}^m\alpha_i y_i^s)^\frac{1}{s}\le\sum_{i=1}^m\alpha_i y_i,
\end{align}
named Jenson's inequality that is stric unless 
$x_1=x_2=...=x_m.$
So, using  (\ref{eqq7}) and (\ref{eq8}) for $0<s<1$ we have
\begin{align}
\prod_{i=1}^my_i^{\alpha_i} \le(\sum_{i=1}^m\alpha_i y_i^s)^\frac{1}{s}\le\sum_{i=1}^m\alpha_i y_i.
\end{align}
\end{remark}

Base on the equality (\ref{eq7}),  if
$G$
be a graph with
$n$
vertices and
$m$ 
edges such that 
$z_i $
for 
$i=1,~ 2, ~..., ~m$
be d-points pertained edges, then for the nonnegative numbers
$|z_1|,~|z_2|,~...,~|z_m|$,
\begin{align}
mR_g\le SO(G).
\end{align}
Equality holds whenever 
$|z_1|=|z_2|=...=|z_m|.$
Also using inequality (\ref{eq8}), consider $\alpha_i=\frac{1}{m}$, $s=\frac{1}{2}$ and $y_i=d(u)^2+d(v)^2$ we have 
\begin{align}
SO(G)\leq \sqrt{F(G)m},
\end{align}
in which 
$F(G)=\Sigma (d(u)^2+d(v)^2).$
The index 
$F(G)$
is called {\it forgotten index} and introduced in \cite{fortulla}. This relation is been proved in \cite{Igor} by another way.

\section{ The geometric mean, variance and Sombor index}
 The variance of the nonnegative real numbers 
 $X=\{x_1^r, x_2^r, \cdots x_m^r\}$
  is dfined as 
$\sigma^2(X^r)=\frac{1}{m}\sum_{i=1}^m(x_i^r-\sum_{i=1}^m\frac{1}{m}x_i^r)^2$
and  displayed as
$\sigma^2(X^r).$

 \begin{theorem}
Let 
$G$
be a graph with 
$n$
vertices and 
$m$
edges
such that for 
$i=1,2,...,m$
$z_i$
be the d-points of 
$G,$
then
\begin{align}
m(R_g+\sigma ^2 (z^{\frac{1}{2}}))\le SO(G)
\end{align}
in which  
$z=\{|z_1|, |z_2|, \cdots, |z_m|\}.$
\end{theorem}
\begin{proof}
With refering to the theorem 1 of the refrence \cite{m-aldas} that states if for 
$i=1,2,...,m$ 
$x_i>0,$
$\alpha_i>0$
and
$\sum_{i=1}^m\alpha_i=1,$
then
\begin{align}
\prod_{i=1}^mx_i^{\alpha_i}\le \sum_{i=1}^m\alpha_i x_i-\sum_{i=1}^m\alpha _i(x_i^\frac{1}{2}-\sum_{k=1}^m\alpha _k x_k^\frac{1}{2})^2.
\end{align}
Now consider the graph 
$G$
with d-points
$z_i$
for 
$i=1,2,...,m.$
So
$|z_i|>0,$
and for 
$\alpha_i=\frac{1}{m}$
we have
\begin{align*}
\prod_{i=1}^m|z_i|^\frac{1}{m}\le \sum_{i=1}^m\frac{1}{m}|z_i|-\sum_{i=1}^m\frac{1}{m}(|z_i|^\frac{1}{2}-\sum_{k=1}^m\frac{1}{m}|z_k|^\frac{1}{2})^2\\ \Rightarrow m\prod_{i=1}^m|z_i|^\frac{1}{m}\le \sum_{i=1}^m|z_i|-\frac{m}{m}\sum_{i=1}^m(|z_i|^\frac{1}{2}-\sum_{k=1}^m\frac{1}{m}|z_k|^\frac{1}{2})^2 \\ \Rightarrow mRg\le SO(G)-m\sigma^2(z^\frac{1}{2})\Rightarrow m(Rg+\sigma^2(z^\frac{1}{2}))\le SO(G).
\end{align*}
\end{proof}

\begin{theorem}
Let 
$G$
be a graph with 
$n$
vertices and
$m$
edges, then for the d-points 
$z_i$
$(i=1,2,...,m)$
put
$M_1=min\{|z_1|,|z_2|,...,|z_m|\}$
and
$M_2=max\{|z_1|,|z_2|,...,|z_m|\},$
then
\begin{align}
m(R_g+\frac{\sigma^2(z)}{2M_2})\le SO(G)\le m(R_g+\frac{\sigma^2(z)}{2M_1}).
\end{align}
\end{theorem}

\begin{proof}
With attention to the remark 3 of the refrence \cite{m-aldas} that if point
$0<M_1=min\{x_1,x_2,...x_m\}$ 
and
$M_2=max\{x_1,x_2,...,x_m\},$
then
\begin{align*}
\frac{1}{2M_2}\sum_{i=1}^m\alpha_i(x_i-\sum_{k=1}^m\alpha_k |x_k|)^2\le \sum_{i=1}^m\alpha_i x_i-\prod_{i=1}^mx_i\alpha_i \le \frac{1}{2M_1}\sum_{i=1}^m\alpha_i(x_i-\sum _{k=1}^m \alpha_k x_k)^2
\end{align*}
Now consider the graph 
$G$
with d-points
$z_i$
for 
$i=1,2,...,m.$
So
$|z_i|>0 $
and for
$\alpha=(\alpha_1,\alpha_2,...,\alpha_m)=(\frac{1}{m},\frac{1}{m},...,\frac{1}{m})$
where
$\sum_{i=1}^m\alpha_i=1,$
we have
\\
\begin{align*}
\frac{1}{2M_2}\sum_{i=1}^m\frac{1}{m}(|z_i|-\sum_{k=1}^m\frac{1}{m}|z_k|)^2\le \sum_{i=1}^m\frac{1}{m} |z_i|-\prod_{i=1}^m|z_i|^\frac{1}{m}\le \frac{1}{2M_1}\sum_{i=1}^m\frac{1}{m}(|z_i|-\sum _{k=1}^m \frac{1}{m} |z_i|)^2
\\ \Rightarrow \frac{m}{2M_2m}\sum_{i=1}^m(|z_i|-\sum_{k=1}^m\frac{1}{m}|z_k|)^2\le \sum_{i=1}^m |z_i|-m\prod_{i=1}^m|z_i|^\frac{1}{m}\le \frac{m}{2M_1m}\sum_{i=1}^m(|z_i|-\sum _{k=1}^m \frac{1}{m}|z_i|)^2
\\ \Rightarrow \frac{m}{2M_2}\sigma^2(z)\le SO(G)-mR_g \le \frac{m}{2M_1}\sigma^2(z),~~~~~~~~~~~~~~~~~~~~~~~~~~~~
\end{align*}
therefore
\begin{align*}
m(R_g+\frac{\sigma ^2(z)}{2M_2})\le SO(G)\le m(R_g+\frac{\sigma ^2(z)}{2M_1}).
\end{align*}
\end{proof}
        
\section{The geometric mean, standard deviation and Sombor index}        
  Upper and lower bounds that introduced for the sombor index in the theorem below show that the Sombor index has behavior closely similar to the standard deviation index or the linear combination of the geometric mean and standard deviation indices, so in order to the bounds for the Sombor index we can apply inequalities for the standard deviation index.

\begin{definition}
Suppose that
$z_1,~z_2,~...,~z_m$
be a sequence of the degree-points pertained the edges of the graph 
$G,$
then for the nonnegative numbers
$|z_1|,~|z_2|,~...,~|z_m|$
the arithmetic mean 
$\mu$
and variance 
$\sigma ^2$
are denotned as:
\begin{align}
\mu=\frac{1}{m}\sum_{i=1}^{n}|z_i|,~~~~~\sigma ^2=\frac{1}{m}\sum_{i=1}^{m}(|z_i|-\mu)^2.
\end{align}
\end{definition}

\begin{theorem}
Let
$z_1,~z_2,~... ,~z_m$
be a sequence of degree-points pertained edges of the graph 
$G$
for 
$m\ge 2,$
then for the sequence 
$|z_1|,~ |z_2|, ~... ,~ |z_m|$
with 
$SO(G) > 0$
and variance 
$\sigma^2,$
\begin{align}\label{eqs}
\frac{m}{\sqrt{m-1}}\sigma< SO(G)\le m(R_g+\sqrt{m-1}~\sigma)
\end{align}
\end{theorem}
\begin{proof}
Base on the Lemma 1.1 of refrence \cite{Rodin}, whereas all of the terms of the sequence 
$|z_1|,~|z_2|,~...,~|z_m|$
are posetive, so
\begin{align*}
\frac{\sigma}{\mu} <\sqrt{m-1}
\end{align*}
but
$\mu=\frac{1}{m}\sum_{i=1}^{n}|z_i|=\frac{SO(G)}{m}$
and so, after replacing in the above inequation, the left inequation (\ref{eqs}) is obtained. For the other hand, we can apply the corallary 1 in refrence \cite{Rodin}, whereas all of the terms of the sequence 
$|z_1|,~|z_2|,~...,~|z_m|$
are posetive, so
\begin{align*}
\mu-\sqrt[m]{|z_1|.|z_2|...|z_3|}\le \sqrt{m-1}~\sigma
\end{align*}
but
$\mu=\frac{1}{m}\sum_{i=1}^{n}|z_i|=\frac{SO(G)}{m}$
and so, after replacing in the above inequation, the right inequation is obtained. 
\end{proof}

\begin{theorem}
Let 
$G$
be a graph with
$n$
vertices,
$m$
edges
and 
$z_i$
for
$i=1,2,...,m$
be its d-points, then
\begin{align}
\frac{1}{\beta_{max}}(\sum_{i=1}^m\beta_i |z_i|-\prod_{i=1}^m|z_i|^{\beta_i})+mR_g \le SO(G) \le \frac{1}{\beta_{min}}(\sum_{i=1}^m\beta_i |z_i|-\prod_{i=1}^m|z_i|^{\beta_i})+mR_g~~~
\end{align}
such that 
$\beta=(\beta_1,\beta_2,...,\beta_m),$
$\beta_i>0$
and
$\sum_{i=1}^m\beta_i=1.$
\\
Equality holds in each of the inequalities if and only if 
$|z_1|=|z_2|=...=|z_m|$
or
$\beta_{max}=\frac{1}{m}$
(or
$\beta_{min}=\frac{1}{m}$).
\end{theorem}
\begin{proof}
With attention to the theorem 2.1 of the reference \cite{Aldas} that state if
$x_i\ge 0,$
$\alpha_i>0, \beta_i>0$
$i=1,2,...,m$
such that 
$\sum_{i=1}^m\alpha_i=\sum_{i=1}^m\beta_i=1,$
$\alpha_{min}:=\{\alpha_1,...,\alpha_m\},$
$\alpha_{max}=max\{\alpha_1,...\alpha_m\},$
and similarity for 
$\beta_{min}$
and
$\beta_{max},$
then
\begin{align*}
\frac{\alpha_{min}}{\beta_{max}}(\sum_{i=1}^m\beta_ix_i-\prod_{i=1}^mx_i^{\beta_i})\le \sum_{i=1}^m\alpha_i x_i-\prod_{i=1}^m|z_i|^{\alpha_i}\le \frac{\alpha_{max}}{\beta_{min}}(\sum_{i=1}^m\beta_i x_i-\prod_{i=1}^mx_i^{\beta_i}).
\end{align*}
Writing 
$\alpha=(\alpha_1,...,\alpha_m)=(\frac{1}{m},\frac{1}{m},...\frac{1}{m}),$
\begin{align*}
\frac{1}{m\beta_{max}}(\sum_{i=1}^m\beta_i |z_i|-\prod_{i=1}^m|z_i|^{\beta_i})\le \sum_{i=1}^m\frac{1}{m} |z_i|-\prod_{i=1}^m|z_i|^{\frac{1}{m}}\le \frac{1}{m\beta_{min}}(\sum_{i=1}^m\beta_i |z_i|-\prod_{i=1}^m|z_i|^{\beta_i}),
\end{align*}
in result
\begin{align*}
\frac{1}{\beta_{max}}(\sum_{i=1}^m\beta_i |z_i|-\prod_{i=1}^m|z_i|^{\beta_i})\le SO(G)-mR_g \le \frac{1}{\beta_{min}}(\sum_{i=1}^m\beta_i|z_i|-\prod_{i=1}^m|z_i|^{\beta_i}),
\end{align*}
therefore
\begin{align*}
\frac{1}{\beta_{max}}(\sum_{i=1}^m\beta_i |z_i|-\prod_{i=1}^m|z_i|^{\beta_i})+mR_g\le SO(G)\le \frac{1}{\beta_{min}}(\sum_{i=1}^m\beta_i |z_i|-\prod_{i=1}^m|z_i|^{\beta_i})+mR_g.
\end{align*}
\end{proof}

\begin{theorem}
Let 
$G$
be a graph with
$n$
vertices and 
$m$
edges, then  for the ratio
$r_m(z)=\frac{mR_g}{SO(G)},$
\begin{align}
\lim_{m\rightarrow \infty}r_m(z)=e^{-\gamma},
\end{align}
in which 
$\gamma$
is named Euler's constant and
$e^{-\gamma}\approx 0.5615.$
\end{theorem}
\begin{proof}
In the refrences \cite{Aldas} (remark 3) and \cite{glas} (Theorem 5.1) stated and proved that for
$n\ge2,$
$x=(x_1,x_2,...,x_m)\in \mathbb{R}^m\setminus \{0\}$
and base on the inequation (\ref{eq7}) for the ratio 
$r_m(x)=\frac{\prod_{i=1}^n|x_i|^{\frac{1}{m}}}{m^{-1}\sum_{i=1}^m|x_i|}$
always
$0\le r_m(x)\le 1$
and
$\lim_{m\rightarrow\infty} r_m(x)=e^{-\gamma}\approx 0.5615.$
\\
So for the d-points
$z_i~(i=1,2,...,m)$
and the vector 
$z=(|z_1|,|z_2|,...,|z_m|)$ 
of the graph 
$G,$
we have
\begin{align*}
r_m(x)=\frac{\prod_{i=1}^m|z_i|^{\frac{1}{m}}}{m^{-1}\sum_{i=1}^m|z_i|}=\frac{mR_g}{SO(G)},
\end{align*}
that base on the inequation (\ref{eq7}) always
$0\le r_m(z)\le1$
and
\begin{align*}
\lim_{m\rightarrow \infty}r_m(z)=e^{-\gamma}\approx 0.5615.
\end{align*}
\end{proof}

\section{The arithmetic-harmonic mean and Sombor index}
Also consider {\it the arithmetic-harmonic mean inequality} that for nonnegative real numbers 
$a_1,~a_2,~a_3,~...,~a_m$
is as follow
\begin{align}
R_h=\frac{m}{\sum_{i=1}^m{\frac{1}{a_i}}}\le \frac{1}{m}\sum_{i=1}^{m}a_i=R_a.
\end{align}      
        
Using the inequality of arithmatic-harmonic mean, we have the following theorem.
\begin{theorem}
Let 
$G$
be a graph with
$n$
vertices and
$m$ 
edges such that 
$z_i$
$(i=1,2,...,m)$
be degree-points pertained edges, then for the nonnegative numbers
$|z_1|,~|z_2|,~...,~|z_m|$
\begin{align}
mR_h\le SO(G).
\end{align}
and the equality holds whenever 
$|z_1|=|z_2|=...=|z_m|.$
\end{theorem}

\section{multiple of the Sombor indices of two graphs}

  Fundamental of the acquisition of the many lower bounds to many areas of mathematics is {\it cauchy-schwarz inequality}, in here also we are going to introuduce this inequality and then find a lower bound by it for sombor index of the graph
$G$, we also use {\it Chebyshev's inequality} for abtaining a upper bound for it.
\\
{\it \textbf{cauchy-schwarz inequality}} 
for real numbers 
$x_i$
and 
$y_i$
$(i=1,2, ...,n)$
is stated as
\begin{align}
(\sum_{i=1}^mx_iy_i)^2\le(\sum_{i=1}^mx_i^2).(\sum_{i=1}^my_i^2).
\end{align}

{\it \textbf{Chebyshev's inequality}} for two increasing sequences 
$\{x_i\}_{i=1}^n$
and
$\{y_i\}_{i=1}^n$
is stated as 
\begin{align}
(\sum_{i=1}^mx_i)(\sum_{i=1}^my_i)\le m\sum_{i=1}^mx_iy_i.
\end{align}

\begin{theorem}
Let 
 $z_1,~z_2,~,...,~z_m$
 be d-points of graph 
 $G_1$
 such that 
 $|z_1|>|z_2|>...>|z_m|$
 and
 $w_1,~w_2,~...,~w_m$
 be d-points of graph 
 $G_2$
 such that
 $|w_1|>|w_2|>...>|w_m|,$
 then
 \begin{align}
 \sum_{i=1}^m|z_i||w_i|\le SO(G_1).SO(G_2)\le m\sum_{i=1}^m|z_i||w_i|.
 \end{align}
 \end{theorem}

\begin{proof}
For the left-hand inequality, consider
$G_1$
and put 
$a_i=\sqrt[4]{d^2(u)+d^2(v)},$
then
$|z_i|=a_i^2=(\sqrt[4]{d^2(u)+d^2(v)})^2=\sqrt{d^2(u)+d^2(v)}$
and in result
$\sum_{i=1}^m|z_i|=\sum_{i=1}^ma_i^2=SO(G_1)$
and for
$G_2$
put
$b_i=\sqrt[4]{d^2(\overline{u})+d^2(\overline{v})},$
then
$|w_i|=b_i^2=(\sqrt[4]{d^2(\overline{u})+d^2(\overline{v})})^2=\sqrt{d^2(u)+d^2(v)}$
and in result
$\sum_{i=1}^m|w_i|=\sum_{i=1}^mb_i^2=SO(G_2),$
so base on the above relations and cauchy-schwarz inequality, we have
\begin{align*}
\sum_{i=1}^m|z_i||w_i|\le (\sum_{i=1}^m(|z_i||w_i|)^{\frac{1}{2}})^2=(\sum_{i=1}^m|z_i|^{\frac{1}{2}}|w_i|^{\frac{1}{2}})^2\\=(\sum_{i=1}^ma_ib_i)^2\le (\sum_{i=1}^ma_i^2)(\sum_{i=1}^mb_i^2)~~~~~~~~~~~~~~~~~~~~~~~~~~~~~~~~~~~~~~~~~~~~\\=(\sum_{i=1}^m\sqrt{d^2(u)+d^2(v)})(\sum_{i=1}^m\sqrt{d^2(u)+d^2(v)})=SO(G_1).SO(G_2).
\end{align*}
For the right-hand inequality, by Chebyshev's inequality, it is concluded that for nonnegative numbers 
$|z_i| $
and
$|w_i|~ (i=1,~2,~...,~m)$
if 
$|z_1|>|z_2|>...>|z_m|$
and 
$|w_1|>|w_2|>...>|w_m|,$
then
\begin{align*}
(\sum_{i=1}^m|z_i|)(\sum_{i=1}^m|w_i|)\le m\sum_{i=1}^m|z_i||w_i|\Rightarrow SO(G_1).SO(G_2)\le m\sum_{i=1}^m|z_i||w_i|.
\end{align*}
\end{proof}

\section{The arithmetic-geometric and geometric-arithmetic indices and Sombor index}

Several years ago, the arithmetic-geometric and geometric-arithmetic indices were introduced in mathematical chemistry. Now in this paper, we determine some boundes for the Sombor index by coefficients of each one of these indices.
\begin{definition}
Two sample of the topological indices applied in the chemical graph theorem that are stated in refrences \cite{cui} and \cite{vuj} are  arithmetic-geometry index and geometric-arithmetic index which are defined as follow
\begin{align}
AG(G)=\sum_{uv\in E(G)}\frac{1}{2}(\sqrt{\frac{d(u)}{d(v)}}+\sqrt{\frac{d(v)}{d(u)}})=\sum_{uv\in E(G)}\frac{1}{2}\big(\frac{d(u)+d(v)}{\sqrt{d(u)d(v)}}\big),
\end{align}
and
\begin{align}
GA(G)=\sum_{uv\in E(G)}\big(\frac{2\sqrt{d(u)d(v)}}{d(u)+d(v)}\big).
\end{align}
\end{definition}

Considering the reference \cite{vuj}, it is proved that
\begin{align}
GA(G)\le AG(G).
\end{align}
Also note that 
\begin{align}
AG(G)\le SO(G),
\end{align}
because for each the edge
$e=uv \in E(G)$
\begin{align*}
\frac{d(u)+d(v)}{\sqrt{2}}\le \sqrt{d^2(u)+d^2(v)}\Rightarrow \frac{d(u)+d(v)}{2}\le \sqrt{d^2(u)+d^2(v)}\\
\\ \Rightarrow  \frac{d(u)+d(v)}{2\sqrt{d(u)d(v)}}\le \sqrt{d^2(u)+d^2(v)}~~~~~~~~~~~~~~~~~~~~~\\ 
\\ \Rightarrow  \sum_{i=1}^m\frac{d(u)+d(v)}{2\sqrt{d(u)d(v)}}\le \sum_{i=1}^m\sqrt{d^2(u)+d^2(v)}~~~~~~~~~~~~\\
\\ \Rightarrow AG(G)\le SO(G).~~~~~~~~~~~~~~~~~~~~~~~~~~~~
\end{align*}
 Now we are going to improve the above bound:

\begin{theorem}
Let 
$G$
be a graph with
$n$
vertices,
$m$ 
edges and with maximum degree 
${\Delta}$
and minimum degree 
$\delta,$
then
\begin{align}
\sqrt{2}\delta AG(G)\le SO(G)\le\sqrt{2}(n-1)AG(G).
\end{align}
\end{theorem}
\begin{proof}
For obtaining lower bound and upper bound, consider the following function 
\begin{align*}
f(x,y)=\frac{\sqrt{x^2+y^2}}{\frac{x+y}{2\sqrt{xy}}}=\frac{2\sqrt{x^3y+y^3x}}{x+y}
\end{align*}
where 
$2\le \delta \le x\le y\le n-1.$
Therefore
\begin{align*}
\frac{\partial f}{\partial x}=\frac{(x^3y+y^3x)(3x^2y+y^3)(x+y)-2\sqrt{x^3y+y^3x}}{(x+y)^2}\ge 0
\end{align*}
and this means that 
$f(x,y)$
is a increasing function in x. Thus, the function obtains its minimum at point 
$(\delta,y_1)$
for some 
$y_1$
such that 
$\delta \le y_1\le n-1$
and its maximum at point 
$(y_2,y_2)$
for some 
$y_2$
such that 
$\delta \le y_2.$
it is concluded that 
\begin{align*}
f(\delta, \delta)\le f(x,y)\le f(n-1,n-1)~~~~~~~~~~~~~~~\\ 
\\ \Rightarrow \sqrt{2}\delta \le f(x,y) \le \sqrt{2}(n-1)~~~~~~~~~~~~~~~~~~\\ 
\\ \Rightarrow  \sqrt{2}\delta \frac{x+y}{2\sqrt{xy}}\le \sqrt{x^2+y^2} \le \sqrt{2}(n-1)\frac{x+y}{2\sqrt{xy}}\\ 
\\ \Rightarrow \sqrt{2}\delta \sum_{i=1}^m\frac{x+y}{2\sqrt{xy}}\le \sum_{i=1}^m\sqrt{x^2+y^2} \le \sqrt{2}(n-1)\sum_{i=1}^m\frac{x+y}{2\sqrt{xy}}\\ 
\\ \Rightarrow \sqrt{2}\delta AG(G) \le SO(G)\le \sqrt{2}(n-1)AG(G).
\end{align*}
\end{proof}

\begin{theorem}
Let
$G$
be a graph with
$n$
vertices and
$m$ 
edges, then
\begin{align}
GA(G)< SO(G).
\end{align}
\end{theorem}
\begin{proof}
Base on the inequality (\ref{eq7}) for each edge 
$e=uv \in E(G),$ 
\begin{align*}
\sqrt{d(u).d(v)}\le \frac{d(u)+d(v)}{2}\Rightarrow 2\sqrt{d(u).d(v)}\le d(u)+d(v),
\end{align*}
then
\begin{align}\label{eq12}
\Rightarrow \frac{2\sqrt{d(u).d(v)}}{d(u)+d(v)}\le 1,
\end{align}
and since 
$\sqrt{d^2u+d^2v}>1,$
\begin{align*}
\sum_{uv \in E(G)}\frac{2\sqrt{d(u).d(v)}}{d(u)+d(v)}<\sum_{uv \in E(G)}\sqrt{d^2(u)+d^2(v)}\\
\\ \Rightarrow GA(G)< SO(G).~~~~~~~~~~~~~~~~~~~~~
\end{align*}
\end{proof}
Now we are going to improve this bound:
\begin{theorem}
Let
$G$
be a graph with
$n$
vertices and
$m$ 
edges, then
\begin{align}
\sqrt{2}\delta GA(G)\leq SO(G)\leq \sqrt{2}(n-1)GA(G)
\end{align}
\end{theorem}
\begin{proof}
Consider the following function
\begin{align*}
f(x,y)=(\frac{\sqrt{x^2+y^2}}{\frac{2\sqrt{xy}}{x+y}})^2=\frac{(x^2+y^2)(x+y)^2}{4xy},
\end{align*}
where 
$2\leq \delta \leq x \leq y\leq n-1$
then
\begin{align*}
\frac{\partial f}{\partial x}=\frac{12x^4y+16x^3y^2+8x^2y^3-4y^5}{(4xy)^2}\ge 0,
\end{align*}
this means that 
$f(x,y)$
is a increasing function in 
$x$
and give its minimum at point 
$(\delta, y_1)$
for some
$\delta$
such that 
$2 \leq \delta\le y_1 \le n-1,$
and its maximum at point 
$(n-1,n-1)$.
it is concluded that
\begin{align*}
f(\delta , \delta)\le f(x,y)\le f(n-1,n-1) \Rightarrow \sqrt{2}\delta \le f(x,y)\le \sqrt{2}(n-1)\\
\Rightarrow \sqrt{2}\delta ~\frac{2\sqrt{xy}}{x+y}\le \sqrt{x^2+y^2}\le \sqrt{2}(n-1)~\frac{2\sqrt{xy}}{x+y}~~~~~~~~~~~~~~~~~~~~\\
\Rightarrow \sqrt{2}\delta  \sum_{i=1}^m\frac{2\sqrt{xy}}{x+y}\le \sum_{i=1}^m\sqrt{x^2+y^2}\le \sqrt{2}(n-1)\sum_{i=1}^m\frac{2\sqrt{xy}}{x+y}~~~~~~\\
\Rightarrow \sqrt{2}\delta GA(G)\le SO(G)\le \sqrt{2}(n-1)GA(G).~~~~~~~~~~~~~~~~~~~~~~~~
\end{align*}
\end{proof}
\section{The symmetric division deg index and Sombor index}
   Base on the reference \cite{vas}, the symmetric division deg index (SDD(G)) is as another index for predicting some physicochemical properties of substances and its tests are carried out by International Academy of Mathematical Chemistry, so here we want to pay a little attention to it and compare it with Sombor index:

\begin{definition}
The symmetric division deg index (SDD(G)) for graph 
$G$
is defined as
\begin{align*}
SDD(G)=(\frac{min\{d(u),d(v)\}}{max\{d(u),d(v)\}}+\frac{max\{d(u),d(v)\}}{min\{d(u),d(v)\}})=\sum_{i=1}^m(\frac{d(u)}{d(v)}+\frac{d(v)}{d(u)}),
\end{align*}
so
\begin{align}
SDD(G)=\sum_{i=1}^m\frac{d^2(u)+d^2(v)}{d(u)d(v)}
\end{align}
\end{definition}

\begin{theorem}
Let
$G$
be a graph with
$n$
vertices and
$m$ 
edges, then
\begin{align}
\frac{\sqrt{2}}{2}\delta SDD(G)\le SO(G)\le \frac{\sqrt{2}}{2}(n-1)SDD(G).
\end{align}
\end{theorem}
\begin{proof}
  For the other hand side of the inequality, consider the following function
\begin{align*}
f(x,y)=(\frac{\sqrt{x^2+y^2}}{\frac{x^2+y^2}{xy}})=\frac{xy\sqrt{x^2+y^2}}{x^2+y^2},
\end{align*}
where 
$2\leq \delta \leq x \leq y\leq n-1$
then
\begin{align*}
\frac{\partial f}{\partial x}=\frac{(y(x^2+y^2)+x)\sqrt{x^2+y^2}-2x^2y\sqrt{x^2+y^2}}{(x^2+y^2)^2}\ge 0,
\end{align*}
this means that 
$f(x,y)$
is a increasing function in 
$x$
and give its minimum at point 
$(\delta, y_1)$
for some
$\delta$
such that 
$2 \le \delta\le y_1 \le n-1,$
and its maximum at point 
$(n-1,n-1)$.
it is concluded that
\begin{align*}
f(\delta , \delta)\le f(x,y)\le f(n-1,n-1) \Rightarrow \frac{\sqrt{2}}{2}\delta \le f(x,y)\le \frac{\sqrt{2}}{2}(n-1)\\
\Rightarrow \frac{\sqrt{2}}{2}\delta \frac{x^2+y^2}{xy}\le \sqrt{x^2+y^2}\le \frac{\sqrt{2}}{2}(n-1)\frac{x^2+y^2}{xy}~~~~~~~~~~~~~~~~~~\\
\Rightarrow \frac{\sqrt{2}}{2}\delta \sum_{i=1}^m\frac{x^2+y^2}{xy}\le \sum_{i=1}^m\sqrt{x^2+y^2}\le \frac{\sqrt{2}}{2}(n-1)\sum_{i=1}^m\frac{x^2+y^2}{xy}~~~\\
\Rightarrow \frac{\sqrt{2}}{2}\delta SDD(G)\le SO(G)\le \frac{\sqrt{2}}{2}(n-1)SDD(G).~~~~~~~~~~~~~~~~~~
\end{align*}
\end{proof}

\section{Conclusion}
In this paper, we provide some bounds for the Sombor index by some other indices like the arithmetic index and geometric index and could  correlate these indices with the variance and standard deviation indices which are important indices in statistic science. Also with using arithmetic-geometric or geometric-arithmetic indices, we determine some bounds for Sombor index base on the coefficients of these indices.

\section*{Data Availability}
The data used to support the findings of this study are included
within the article.

\section*{Conflicts of Interest}
The authors declare that they have no conflicts of interest.

%
%
%
%
%
%
%


\end{document}